\documentclass{amsart}
\usepackage{latexsym, graphicx, amsmath, amssymb, cite, hyperref}

\newtheorem{theorem}{Theorem}[section]
\newtheorem{lemma}[theorem]{Lemma}

\theoremstyle{definition}
\newtheorem{definition}[theorem]{Definition}

\theoremstyle{remark}
\newtheorem{remark}[theorem]{Remark}

\numberwithin{equation}{section}

\title{The VC-dimension and point configurations in ${\Bbb F}_q^2$}
\author{D. Fitzpatrick} 
\author{A. Iosevich}
\author{B. McDonald} 
\author{E. L. Wyman}

\thanks{The second listed author's research was supported in part by the National Science Foundation grant no. HDR TRIPODS - 1934962. The fourth listed author's research was supported in part by the 2021 Simons Travel Grant.}

\begin{document}

\begin{abstract} Let $X$ be a set and ${\mathcal H}$ a collection of functions from $X$ to $\{0,1\}$. We say that ${\mathcal H}$ shatters a finite set $C \subset X$ if the restriction of ${\mathcal H}$ yields every possible function from $C$ to $\{0,1\}$. The VC-dimension of ${\mathcal H}$ is the largest number $d$ such that there exists a set of size $d$ shattered by ${\mathcal H}$, and no set of size $d+1$ is shattered by ${\mathcal H}$. Vapnik and Chervonenkis introduced this idea in the early 70s in the context of learning theory, and this idea has also had a significant impact on other areas of mathematics. In this paper we study the VC-dimension of a class of functions ${\mathcal H}$ defined on ${\Bbb F}_q^d$, the $d$-dimensional vector space over the finite field with $q$ elements. Define 
$$ {\mathcal H}^d_t=\{h_y(x): y \in {\Bbb F}_q^d \},$$ where for $x \in {\Bbb F}_q^d$, $h_y(x)=1$ if $||x-y||=t$, and $0$ otherwise, where here, and throughout, $||x||=x_1^2+x_2^2+\dots+x_d^2$. Here $t \in {\Bbb F}_q$, $t \not=0$. Define ${\mathcal H}_t^d(E)$ the same way with respect to $E \subset {\Bbb F}_q^d$. The learning task here is to find a sphere of radius $t$ centered at some point $y \in E$ unknown to the learner. The learning process consists of taking random samples of elements of $E$ of sufficiently large size. 

We are going to prove that when $d=2$, and $|E| \ge Cq^{\frac{15}{8}}$, the VC-dimension of ${\mathcal H}^2_t(E)$ is equal to $3$. This leads to an intricate configuration problem which is interesting in its own right and requires a new approach. 
\end{abstract}

\maketitle

\section{Introduction} 

\vskip.125in 

The purpose of this paper is to study the Vapnik-Chervonenkis dimension in the context of a naturally arising family of functions on subsets of the two-dimensional vector space over the finite field with $q$ elements, denoted by ${\Bbb F}_q^d$. Let us begin by recalling some definitions and basic results (see e.g. \cite{DS14}, Chapter 6). 

\begin{definition} \label{shatteringdef} Let $X$ be a set and ${\mathcal H}$ a collection of functions from $X$ to $\{0,1\}$. We say that ${\mathcal H}$ shatters a finite set $C \subset X$ if the restriction of ${\mathcal H}$ to $C$ yields every possible function from $C$ to $\{0,1\}$. \end{definition} 

\vskip.125in 

\begin{definition} \label{vcdimdef} Let $X$ and ${\mathcal H}$ be as above. We say that a non-negative integer $n$ is the VC-dimension of ${\mathcal H}$ if there exists a set $C \subset X$ of size $n$ that is shattered by ${\mathcal H}$, and no subset of $X$ of size $n+1$ is shattered by ${\mathcal H}$. \end{definition} 

\vskip.125in 

We are going to work with a class of functions ${\mathcal H}^2_t$, where $t \not=0$. Let $X={\Bbb F}_q^2$, and define 
\begin{equation} \label{functionclassdef} {\mathcal H}_t^2=\{h_y: y \in {\Bbb F}_q^2 \}, \end{equation} where $y \in {\Bbb F}_q^2$, and $h_y(x)=1$ if $||x-y||=t$, and $0$ otherwise, where here, and throughout, $||x||=x_1^2+x_2^2$. Let ${\mathcal H}_t^2(E)$ be defined the same way, but with respect to a set $E \subset {\Bbb F}_q^2$ i.e 
$$ {\mathcal H}^2_t(E)=\{h_y: y \in E\},$$ where $h_y(x)=1$ if $||x-y||=t$ ($x \in E$), and $0$ otherwise. 

\vskip.125in 

Our main result is the following. 

\begin{theorem} \label{main} Let ${\mathcal H}^2_t(E)$ be defined as above with respect to $E \subset {\Bbb F}_q^2$, $t \not=0$. If 
$|E| \ge C q^{\frac{15}{8}}$, with a sufficiently large constant $C$, then the VC-dimension of ${\mathcal H}^2_t(E)$ is equal to $3$. \end{theorem} 

\vskip.125in 

\begin{remark} It is interesting to note since $|{\mathcal H}_t^2(E)|=|E|$, it is clear that the VC-dimension of ${\mathcal H}_t^2(E) \leq log_2(|E|)$, so $3$ is a clear improvement over this general estimate. It is not difficult to see that the VC-dimension is $<4$, so the real challenge to establish the $3$ bound. Moreover, our result says that in this sense, the learning complexity of subsets of ${\Bbb F}_q^2$ of size $>Cq^{\frac{15}{8}}$ is the same as that of the whole vector space ${\Bbb F}_q^2$. \end{remark} 

\vskip.125in 

\begin{remark} The higher dimensional case of this problem is somewhat easier from the point of view of the underlying Fourier analytic techniques, but is more complex in terms of geometry. We shall address this issue in a sequel (\cite{GIJMMSW21}). \end{remark} 

\vskip.125in 

We can prove that the VC-dimension is at least $2$ under a much weaker assumption. 

\begin{theorem} \label{mainlame} Let ${\mathcal H}^2_t(E)$ be defined as above with respect to $E \subset {\Bbb F}_q^2$, $t \not=0$. If $|E| \ge Cq^{\frac {3}{2}}$, with a sufficiently large constant $C$, then the VC-dimension of ${\mathcal H}^2_t(E)$ is at least $2$ and no more than $3$. \end{theorem} 

\vskip.125in 

\begin{remark} The discrepancy between the size thresholds in Theorem \ref{main} and Theorem \ref{mainlame} raises the question of whether the VC-dimension is, in general, $<3$, if $|E|$ is much smaller than $Cq^{\frac{15}{8}}$. We do not know the answer to this question and hope to resolve it in the sequel. \end{remark} 

\vskip.125in 

\section{Learning theory perspective on Theorem \ref{main}} 

\vskip.125in 

From the point of view of learning theory, it is interesting to ask what the ``learning task" is in the situation at hand. It can be described as follows. We are asked to construct a function $f:E \to \{0,1\}$, $E \subset {\Bbb F}_q^2$, that is equal to $1$ on a sphere of radius $t$ centered at some $y^{*} \in E$, but we do not know the value of $y^{*}$. The fundamental theorem of statistical learning tells us that if the VC-dimension of ${\mathcal H}_t^2(E)$ is finite, we can find an arbitrarily accurate hypothesis (element of ${\mathcal H}^2_t(E)$) with arbitrarily high probability if we consider a randomly chosen sampling training set of sufficiently large size. 

We shall now make these concepts precise. Let us recall some more basic notions.

\begin{definition} Given a set $X$, a probability distribution $D$ and a labeling function $f: X \to \{0,1\}$, let $h$ be a hypothesis, i.e $h:X \to \{0,1\}$, and define 
$$ L_{D,f}(h)={\Bbb P}_{x \sim D}[h(x) \not=f(x)],$$ where ${\Bbb P}_{x \sim D}$ means that $x$ is being sampled according to the probability distribution $D$. 
\end{definition} 

\vskip.125in 

\begin{definition} A hypothesis class ${\mathcal H}$ is PAC learnable if there exist a function 
$$m_{{\mathcal H}}: {(0,1)}^2 \to {\Bbb N}$$ and a learning algorithm with the following property: For every $\epsilon, \delta \in (0,1)$, for every distribution $D$ over $X$, and for every labeling function $f: X \to \{0,1\},$ if the realizability assumption holds with respect to $X$, $D$, $f$, then when running the learning algorithm on $m \ge m_{{\mathcal H}}(\epsilon, \delta)$ i.i.d. examples generated by $D$, and labeled by $f$, the algorithm returns
a hypothesis $h$ such that, with probability of at least $1-\delta$ (over the choice of the examples), 
$$L_{D,f}(h) \leq \epsilon.$$
\end{definition} 

\vskip.125in 

The following theorem is a quantitative version of the fundamental theorem of machine learning, and provides the link between VC-dimension and learnability (see \cite{DS14}).

\begin{theorem} Let ${\mathcal H}$ be a collection of hypotheses on a set $X$. Then ${\mathcal H}$ has a finite VC-dimension if and only if ${\mathcal H}$ is PAC learnable. Moreover, if the VC-dimension of ${\mathcal H}$ is equal to $n$, then ${\mathcal H}$ is PAC learnable and there exist constants $C_1, C_2$ such that 
$$ C_1 \frac{n+\log \left(\frac{1}{\delta} \right)}{\epsilon} \leq m_{{\mathcal H}}(\epsilon, \delta) \leq C_2 \frac{n \log \left(\frac{1}{\epsilon} \right)+
\log \left(\frac{1}{\delta} \right)}{\epsilon}.$$

\end{theorem} 

\vskip.125in 

Going back to the learning task associated with ${\mathcal H}_t^2(E)$, as in Theorem \ref{main}, suppose that $h_y$ is a ``wrong" hypothesis, i.e 
$y \not=y^{*}$, where $f=h_{y^{*}}$ is the true labeling function. Moreover, assume that 
$$ \{z \in {\Bbb F}_q^2: ||z-y||=t \} \cap \{z \in {\Bbb F}_q^d: ||z-y^{*}||=t \}=\emptyset.$$

Since the size of a sphere of non-zero radius in ${\Bbb F}_q^2$ is $q$ plus lower order terms, and $D$ is the uniform probability distribution on 
${\Bbb F}_q^d$, 
$$ L_{D,f}(h) \leq \frac{1}{q} \left(1+o(1) \right),$$ so one must choose $\epsilon$ just slightly less than $\frac{1}{q}$ to make the results meaningful. It follows by taking 
$\delta=\epsilon$ that we need to consider random samples of size $\approx C q \log(q)$ with sufficiently large $C$ to execute the desired algorithm. Moreover, since $3$ points determine a circle effectively means that if $\epsilon$ is just slightly less than $\frac{1}{q}$, then $L_{D,f}(h)=0$. 

\vskip.25in 

\section{Proof of Theorem \ref{mainlame}}

\vskip.25in 


We warm up to Theorem \ref{mainlame} by first showing the VC-dimension of ${\mathcal H}^2_t(E)$ is at least $1$.

The existence of a set of size $1$ that is shattered by ${\mathcal H}^2_t(E)$ means that there exists $x \in E$ with the property that there exist $y \in E$ such that $||x-y||=t$, and $y' \in E$ such that $||x-y|| \not=t$. To find $x$ and $y$, we require a result of the first listed author and Misha Rudnev \cite{IR07}, stated below for convenience.

\begin{theorem}\label{IR07} (\cite{IR07}) \label{IR07}Let $E \subset {\Bbb F}_q^d$, $d \ge 2$. Then if $t \not=0$, 
\begin{equation} \label{irformula} |\{(x,y) \in E: ||x-y||=t \}|={|E|}^2q^{-1}+{\mathcal D}_t(E), \end{equation} where 
$$ |{\mathcal D}_t(E)| \leq 2q^{\frac{d-1}{2}}|E|.$$ 
In particular, if $|E|> 2q^{\frac{d+1}{2}}$, the left hand side of (\ref{irformula}) is positive. Moreover, if $|E| \geq 4 q^{\frac{d+1}{2}}$, then the left hand side of (\ref{irformula}) is $\ge \frac{{|E|}^2}{2q}$. 
\end{theorem} 

By Theorem \ref{IR07}, since $|E| \ge 4q^{\frac{3}{2}}$, there exist $x,y \in E$ such that $||x-y||=t$. Since $|E|$ is much greater than $q$, there also exists $y'$ such that $||x-y'|| \not=t$. Hence, the VC-dimension of ${\mathcal H}^2_t(E)$ is at least $1$.

%

\vskip.125in 


To prove Theorem \ref{mainlame}, we must show that there exists $\{x^1, x^2\} \subset E$ that is shattered by ${\mathcal H}^2_t(E)$. This means that there exist $y^1, y^2, y^{12}, y^{0} \in E$ such that the following hold: \begin{itemize} 
\item i) $||x^1-y^{12}||=||x^2-y^{12}||=t$. 
\item ii) $||x^1-y^1||=t$, $||x^2-y^1|| \not=t$. 
\item iii) $||x^2-y^2||=t$, $||x^1-y^2|| \not=t$. 
\item iv) $||x^1-y^{0}|| \not=t$, $||x^2-y^0|| \not=t$. 
\end{itemize} 

\vskip.125in 

Thus proving the existence of a set $\{x^1,x^2\}$ that is shattered by ${\mathcal H}^2_t(E)$ amounts to establishing the existence of a chain $z^1, z^2, z^3, z^4, z^5 \in E$, such that $||z^{j+1}-z^j||=t$, $j=1,2,3,4$, $||z^1-z^4|| \not=t$, $||z^2-z^5|| \not=t$. Here, $x^1 = z^2$, $x^2 = z^4$, $y^{12} = z^3$, $y^{1} = z^1$, and $y^2 = z^5$ (see figure \ref{vcdim2chain}). Since $|E| \gg q$, we may select $y^0$ from $E$ outside the union of the circles of radius $t$ centered at $x^1$ and $x^2$.
\begin{figure}
\includegraphics[width=0.75\textwidth]{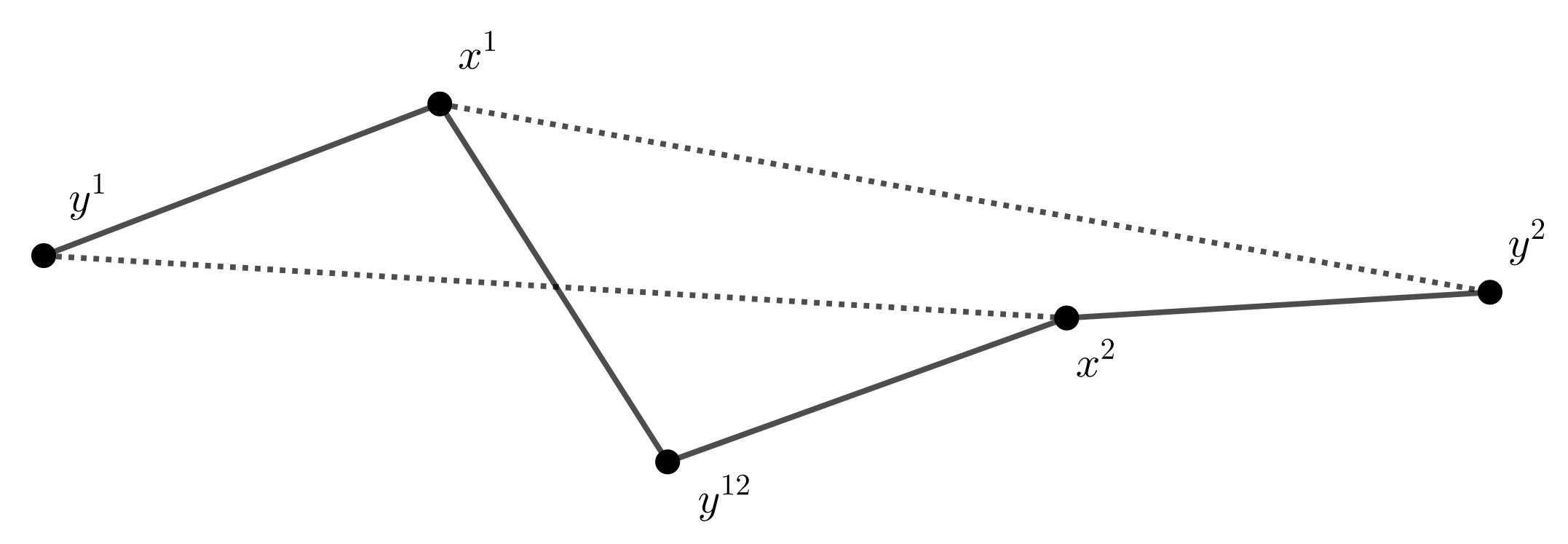}
\caption{Points adjoined by a solid line are separated by a distance $t$, and those joined by a dotted line are separated by a distance $\neq t$.}
\label{vcdim2chain}
\end{figure}

We shall need the following result due to Bennett, Chapman, Covert, Hart, the first listed author, and Pakianathan (\cite{BCCHIP16}, Theorem 1.1). 

\begin{theorem}[\cite{BCCHIP16}] \label{BCCHIP16} Let $E \subset {\Bbb F}_q^d$, $d \ge 2$, and $|E|>\frac{2k}{\log(2)} q^{\frac{d+1}{2}}$. Suppose that $t_i \not=0$, $1 \leq i \leq k$, and let $\vec{t}=(t_1, \dots, t_k)$. Define 
$$ {\mathcal C}_k(\vec{t})=|\{(x^1, \dots, x^{k+1}) \in E^{k+1}: ||x^i-x^{i+1}||=t_i, 1 \leq i \leq k\}|.$$ 
Then 
$$ {\mathcal C}_k(\vec{t})=\frac{{|E|}^{k+1}}{q^k}+{\mathcal D}_k(\vec{t}),$$ where 
$$ |{\mathcal D}_k(\vec{t})| \leq \frac{2k}{\log(2)} q^{\frac{d+1}{2}} \frac{{|E|}^k}{q^k}.$$ 
\end{theorem} 

\vskip.125in 

The existence of a chain of length $4$ ($4$ edges and $5$ vertices) with gap $t \not=0$ follows from this immediately, provided that $|E| \ge Cq^{\frac{3}{2}}$, but we need to work a bit to make sure that we can find such a chain with $||z^{j+1}-z^j||=t$, $j=1,2,3,4$, $||z^1-z^4|| \not=t$, $||z^2-z^5|| \not=t$. To this end, we are going to show that 
\begin{equation} \label{degeneratechain}
	|\{(z^1,z^2,z^3,z^4,z^5) \in E^5: ||z^{i+1}-z^i||=t, 1 \leq i \leq 4, ||z^1-z^4||=t \}| \leq C'{|E|}^3 q^{-1} \end{equation}
if $|E| \ge Cq^{\frac{3}{2}}$. This suffices since by Theorem \ref{BCCHIP16}, 
$$
	|\{(z^1,z^2,z^3,z^4,z^5) \in E^5: ||z^{i+1}-z^i||=t, 1 \leq i \leq 4 \}| \ge C'' {|E|}^5q^{-4},
$$
and, by $|E| \geq C q^{\frac32}$. 

To prove (\ref{degeneratechain}), observe that the left hand side of (\ref{degeneratechain}) is equal to 
\begin{multline*}
\sum_{x,y,z} E(x)E(y) {\left( \sum_u E(u) S_t(x-u) S_t(y-u) \right)}^2 E(z) S_t(y-z) \\
= \sum_{x,y,z; x \not=y} E(x)E(y) {\left( \sum_u E(u) S_t(x-u) S_t(y-u) \right)}^2 E(z) S_t(y-z) \\
+ \sum_{x,u,v,z} E(x)E(u)E(v)E(z)S_t(x-u) S_t(x-v)S_t(x-z)=I+II.
\end{multline*} 

It is not difficult to see that 

\begin{equation} \label{almost} I \leq 2 \sum_{x,y,u,z} E(x)E(y) E(u) E(z) S_t(x-u) S_t(y-u) S_t(y-z) \end{equation} since $x \not=y$ and two circles intersect at at most two points. On the other hand, 
\begin{equation} \label{almost2} II \leq (q+1) \sum_{x,u,v} E(x)E(u)E(v) S_t(x-u) S_t(x-v) \end{equation} since 
$$ \sum_z E(z) S_t(y-z) \leq |S_t| \leq q+1.$$ 

The expression (\ref{almost}) is $\leq C'{|E|}^4q^{-3} \leq C'|E|^3 q^{-1}$ if $|E| \ge Cq^{\frac{3}{2}}$ by Theorem \ref{BCCHIP16} above, and the expression (\ref{almost2}) is $\leq C'' {|E|}^3 q^{-1}$ if $|E| \ge Cq^{\frac{3}{2}}$, also by Theorem \ref{BCCHIP16}, so the claim is proved and we have established that the VC-dimension is at least two.

\vskip.125in

To show the VC-dimension is at most $3$, we claim no subset $\{x^1,x^2,x^3,x^4\}$ of size $4$ can be shattered by ${\mathcal H}_t^2({\Bbb F}_q^2)$, let alone ${\mathcal H}_t^2(E)$. If there were, all four points would be forced to live on the same circle centered at $y^{1234}$, say, and at the same time, there must exist $y^{123}$ such that $x^1,x^2,x^3$ live on a circle of radius $t$ centered at $y^{123}$, while $x^4$ does not. This is impossible since three points determine a circle.

\vskip.25in

\section{Proofs of Theorem \ref{main}}

We know already the VC-dimension of ${\mathcal H}^2_t(E)$ is at most $3$ from the argument in the previous paragraph. Now we must show that there exists $C$ of size $3$ that is shattered by ${\mathcal H}^2_t(E)$. This leads to the following question. Do there exist $x^1,x^2,x^3, y^{123}, y^{12}, y^{13}, y^{23}, y^1, y^2, y^3, y^0 \in E$, such that $||x^i-y^{123}||=||x^i-y^{ij}||=||x^i-y^i||=t$, $i,j=1,2,3$, and all the remaining pair-wise distances between $x$'s and $y$'s do not equal $t$? 

There are several results in literature that prove the existence of a general point configurations in ${\Bbb F}_q^d$ inside sufficiently large sets. Let $G$ be a graph and let $t \not=0$ be given. We say that $G$ can be embedded in $E \subset {\Bbb F}_q^d$, if there exist $x^1, \dots, x^{k+1} \in E$ such that $||x^i-x^j||=t$ for $(i,j)$ corresponding to the pairs of vertices connected by edges in $G$. The second listed author and Hans Parshall proved in \cite{IP19} that if the maximum vertex multiplicity in $G$ is equal to $t$ and $|E| \ge Cq^{\frac{d-1}{2}+t}$, $E \subset {\Bbb F}_q^d$, $d \ge 2$, then $G$ can be embedded in $E$. In the case of the configuration above, $t=4$, so the threshold exponent in \cite{IP19} is $\frac{1}{2}+4>2$, so very different methods are required in this situation. 

\vskip.125in 

%

We shall need the following existence lemma for rhombi. 

\begin{lemma} \label{4loops} Suppose that $|E| \ge 4q^{\frac{7}{4}}$, $t \not=0$, and $v$ is a non-zero vector in ${\Bbb F}_q^2$. Then there exist distinct $x,y,z,w \in E$ such that 
$$||x-y||=||y-z||=||z-w||=||w-x||=t,$$ and neither $x-y$ nor $y-z$ is equal to $\pm v$. 
\end{lemma} 

\vskip.125in 

\begin{proof}
We first claim that less than half the pairs in $\{(x,y) \in E \times E : \|x - y\| = t\}$ satisfy $x - y = \pm v$. This follows from
\begin{multline*}
	|\{(x,y) \in E \times E : x - y = \pm v \}| \\ 
	\leq 2|E| < \frac{|E|^2}{4q} \leq \frac12 |\{(x,y) \in E \times E : \|x - y\| = t\}|,
\end{multline*}
where the second inequality follows from $|E| > 8q$, and the third follows from $|E| \geq 4q^\frac32$ and Theorem \ref{IR07}. By pidgeonholing on the remaining directions, there exists $u$ with $\|u\| = t$ and $u \neq \pm v$ for which
\[
	|\{(x,y) \in E \times E : x - y = u \}| \geq \frac{|E|^2}{4q^2}.
\]
Let $E'$ denote the collection of $x$'s from the set above. The hypothesis $|E| \geq 4 q^\frac74$ ensures
\[
	|E'| \geq \frac{|E|^2}{4q^2} \geq 4q^{\frac32},
\]
and so Theorem \ref{IR07} guarantees there are at least $\frac{|E'|^2}{2q^2}$ pairs $(x,w) \in E' \times E'$ with $\|x - w\| = t$. Next, we must ensure $x - w \neq \pm v$ nor $\pm u$. By proceeding as above, we find
\begin{multline*}
	|\{(x,w) \in E' \times E' : x - w = \pm v, \pm u\}| \\
	\leq 4|E'| < \frac{|E'|^2}{2q} \leq |\{(x,w) \in E' \times E' : \|x - w\| = t\}|,
\end{multline*}
where the second inequality follows since $|E'| > 8 q$, and the third follows again from Theorem \ref{IR07}. Hence, there exists some pair $(x,w)$ in the right-hand set but not the left-hand set.

To summarize, we have found $(x,y,w,z) \in E^4$ for which $\|x - y\| = \|x - w\| = \|w - z\| = \|y - z\| = t$, a rhombus. Furthermore, none of these four sides are parallel to $v$ by construction. Finally, all four points are distinct since $w - y = (x - y) - (x - w) = u - (x - w) \neq 0$ and $x - z = (x - w) + (w - z) = (x - w) + u \neq 0$.
\end{proof} 

%
%

\vskip.125in 

%
%
%
%
%

%
%
%
%

We shall also need the following pigeon-holing observation. 

\begin{lemma} \label{pigeon} Let $E$ be as in the statement of Theorem \ref{main}. Then for any nonzero $t\in \mathbb{F}_q$, there exists $v \in {\Bbb F}_q^2$, $||v||=t$, such that 
$$ |E \cap (E-v)| \ge \frac{1}{2}|E|^2q^{-2}.$$ 
\end{lemma} 

\begin{proof} Using Theorem \ref{IR07} once again, we see that if $|E| \ge Cq^{\frac{15}{8}}$, then 
$$ |\{(x,y) \in E \times E: ||x-y||=t \}| \ge \frac{{|E|}^2}{2q}.$$ 
Since the circle of non-zero radius has at most $q+1$ points, the conclusion follows. 
\end{proof}

\vskip.125in 

It follows from the assumptions of Theorem \ref{main} and Lemma \ref{pigeon} that there exists $v\in \mathbb{F}_q^d$, $||v||=t\neq 0$, such that 
\begin{equation} \label{pigeonquantity} |E \cap (E-v)| \ge 4q^{\frac{7}{4}}. \end{equation}
Using Lemma \ref{pigeon} and Lemma \ref{4loops}, we see that there exist distinct $x,y,z,w \in E \cap (E-v)$, with $v$ from (\ref{pigeonquantity}) such that 
$$ ||x-y||=||y-z||=||z-w||=||x-w||=t, $$ where 
$$\pm v \not=x-y, y-z, z-w, x-w.$$
We are now ready to move into the final phase of the proof of Theorem \ref{main}. 

\vskip.125in 

Let $y^{123}=y$, $x^1=x$, $x^2=y+v$, $x^3=z$, $y^{12}=x+v$, $y^{23}=z+v$, $y^{13}=w$. Note that 
$$ ||y^{123}-x^i||=t, \ i=1,2,3,$$ 
$$ ||y^{12}-x^i||=t, \ i=1,2,$$ 
$$ ||y^{13}-x^i||=t, \ i=1,3,$$ 
$$ ||y^{23}-x^i||=t, \ i=2,3.$$
See Figure \ref{vcdim3figure} for reference. To see that $||y^{ij}-x^k||\neq t$ when $k\neq i,j$, note that otherwise the circles of radius $t$ centered at $y^{ijk}$ and $y^{ij}$ would intersect at $x^1,x^2$, and $x^3$, implying that $y^{ij}=y^{ijk}$, contradicting the construction. Since $|E| \gg q$, we may also select $y^0$ for which $\|y^0 - x^i\| \neq t$ for each $i = 1,2,3$.

\begin{figure}
\includegraphics[width=0.9\textwidth]{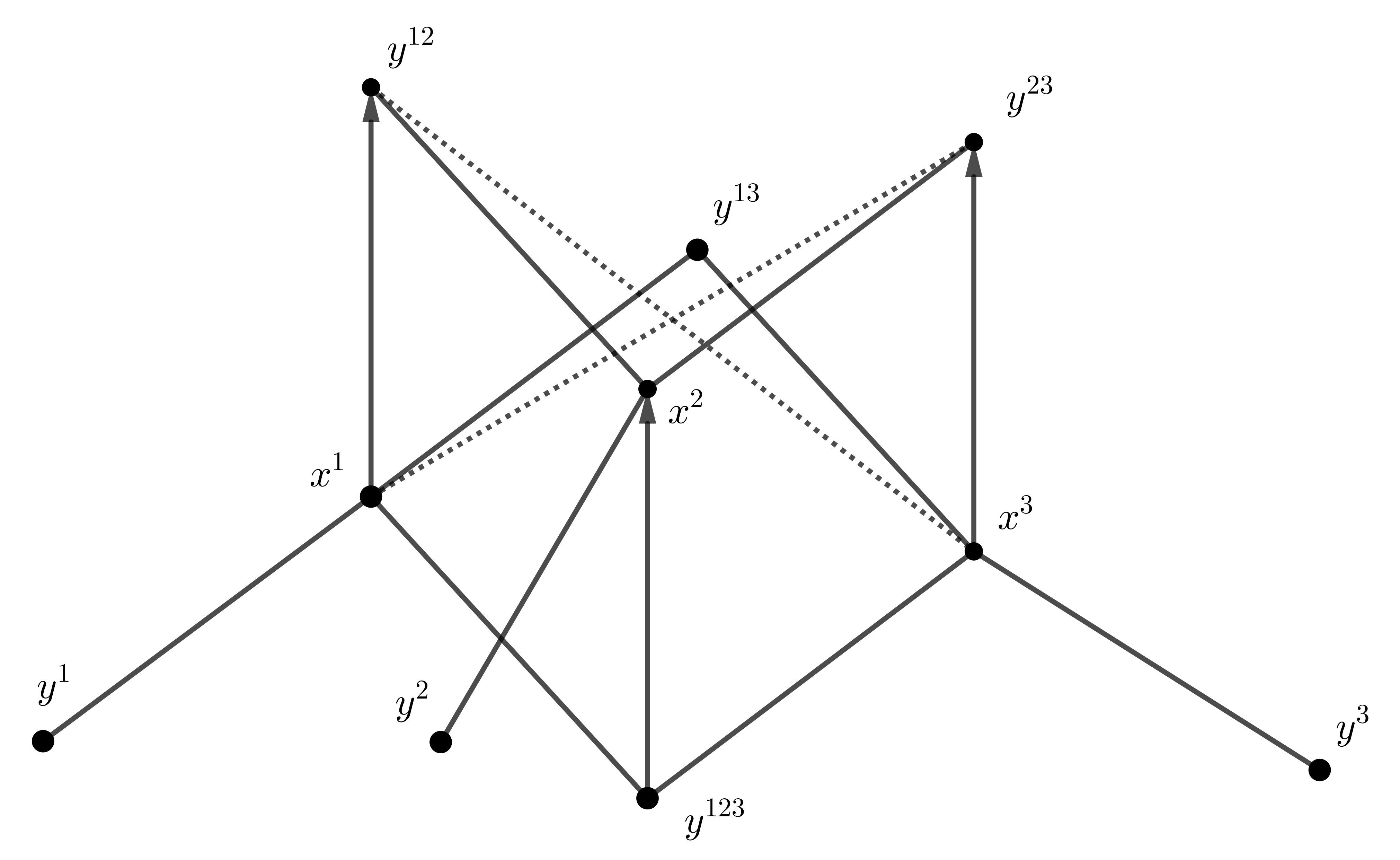}
\caption{Points adjoined by a solid line are separated by a distance $t$, and those joined by a dotted line are separated by a distance $\neq t$. It is not marked by a dotted line, but the distances between points $y^i$ and $x^j$ for $i \neq j$ are $\neq t$. The vertical lines decorated by arrows denote the vector $v$ in the construction.}
\label{vcdim3figure}
\end{figure}

\vskip.125in

We are almost there, but we still need to come up with $y^1, y^2, y^3$ such that $||x^i-y^i||=t$, and we need to make sure that $||x^i-y^j|| \not=t$, $i \not=j$. This is where we now turn our attention. 

\vskip.125in

$E\ast S_t(x)=\left|\{y\in E: \ \|x-y\|=t\}\right|$, and so by Theorem \ref{IR07}, if $|E|>4q^{\frac{3}{2}}$ then

$$
\sum_{x\in E}{E\ast S_t(x)}\geq \frac{1}{2}|E|^2q^{-1}
$$
Moreover, if $|E|>4q^{\frac{3}{2}}$ then $|E|>4\cdot 99q$ for $q\geq 99^2$, and thus $99|E|\leq \frac{1}{4}|E|^2q^{-1}$.  Therefore,
$$
\sum_{x\in E}{E\ast S_t(x)}\leq 99\left|\{x\in E: E\ast S_t(x) \leq 99\}\right|+\sum_{\substack{x\in E \\ E\ast S_t(x)\geq 100}}{E\ast S_t(x)}
$$
$$
\leq \frac{1}{4}|E|^2q^{-1}+\sum_{\substack{x\in E \\ E\ast S_t(x)\geq 100}}{E\ast S_t(x)},
$$
and so 
$$
\sum_{\substack{x\in E \\ E\ast S_t(x)\geq 100}}{E\ast S_t(x)}\geq \frac{1}{4}|E|^2q^{-1}.
$$
By Cauchy-Schwarz,
\begin{multline*}
\frac{1}{16}|E|^4q^{-2}\leq \left(\sum_{\substack{x\in E \\ E\ast S_t(x)\geq 100}}{E\ast S_t(x)}\right)^2 \\
\leq \left|\{x\in E: \ E\ast S_t(x)\geq 100\}\right|\left(\sum_{x\in E}{(E\ast S_t(x))^2}\right).
\end{multline*}
But 
$$
\sum_{x\in E}{(E\ast S_t(x))^2}
=\sum_{x,y,z}{E(y)E(z)S_t(x-y)S_t(x-z)}
$$
is the number of paths of length 2 (2 edges and 3 vertices) in the distance graph of $E$.  By Theorem \ref{BCCHIP16}, if $|E|>\frac{4}{\log{2}}q^{\frac{3}{2}}$ then the number of paths of length 2 is $\leq 2\frac{|E|^3}{q^2}$.  Therefore, 
$$
\frac{1}{16}|E|^4q^{-2}\leq 2\frac{|E|^3}{q^2}\left|\{x\in E: \ E\ast S_t(x)\geq 100\}\right|,
$$
and
$$
\left|\{x\in E: \ E\ast S_t(x)\geq 100\}\right|\geq \frac{1}{32}|E|.
$$
Recall that whenever $|E|\geq Cq^{\frac{15}{8}}$, we have constructed a configuration 
$$\{x^1,x^2,x^3,y^{123},y^{12},y^{13},y^{23},y^0\}$$ with the desired edges in the distance graph (see Figure \ref{vcdim3figure}).  In particular, provided the constant $C$ is large enough, we can construct such a configuration in $E':=\left|\{x\in E: \ E\ast S_t(x)\geq 100\}\right|$, a subset of $E$ in which every vertex has degree at least 100 in the distance graph on $E$.  In particular $x^1,x^2,x^3$ each have degree at least 100, so they each have at least one neighbor in addition to the ones listed, i.e. there exist distinct $y^1,y^2,y^3$ with 
$$
y^1,y^2,y^3\notin \{x^1,x^2,x^3,y^{12},y^{13},y^{23},y^{123}\},
$$
and $||y^i-x^i||=t$ for $i=1,2,3$, and $\|y^i - x^j\| \neq t$ for $i \neq j$.

%


\end{document}